\documentclass{amsart}
\usepackage{amssymb}
\usepackage{amsmath}
\usepackage{textcomp}
\makeindex

\newcommand{\ba}{\begin{array}}
\newcommand{\eea}{\end{eqnarray}}
\newcommand{\ea}{\end{array}}

\newtheorem{definition}{Definition}[section]
\newtheorem{theorem}[definition]{Theorem}
\newtheorem{lemma}[definition]{Lemma}

\newtheorem{remark}[definition]{Remark}

\usepackage[matrix,arrow,curve]{xy}
\usepackage[utf8x]{inputenc}
\usepackage{tikz}

\begin{document}
\title[Weinstein Homotopies]{Weinstein Homotopies}
\author[S. Mukherjee]{Sauvik Mukherjee}
\address{Presidency University, Kolkata, India.\\
e-mail:mukherjeesauvik@gmail.com\\}
\keywords{Weinstein structures, Liouville homotopies}

\begin{abstract} 
In this paper we discuss a problem mentioned by Eliashberg in his paper \cite{Eliashberg}. He has asked if two completed Weinstein structures $(\hat{X},\lambda_0,\phi_0)$ and $(\hat{X},\lambda_1,\phi_1)$ on the same symplectic manifold $(\hat{X},\omega)$ can be homotoped through Weinstein structures. We discuss this problem and prove a weak partial result by assuming some additional conditions. 
 \end{abstract}
\maketitle

\section{introduction}
\label{1}
 In this paper we discuss a problem mentioned by Eliashberg in his paper \cite{Eliashberg}. He has asked if two completed Weinstein structures $(\hat{X},\lambda_0,\phi_0)$ and $(\hat{X},\lambda_1,\phi_1)$ on the same symplectic manifold $(\hat{X},\omega)$ can be homotoped through Weinstein structures. We discuss this problem and prove a weak partial result by assuming some additional conditions.\\

We begin with the basic definitions. Let $(X,\omega)$ be a $2n$-dimensional symplectic domain with boundary with an exact symplectic form $\omega$ and primitive form $\lambda$ i.e, $d\lambda=\omega$.\\

A Liouville form is a choice of a primitive form $\lambda$ such that $\lambda_{\mid \partial X}$ is a contact form on $\partial X$ and the orientation on $\partial X$ by the form $\lambda \wedge d\lambda^{n-1}_{\mid \partial X}$ coincides with its orientation as the boundary of $(X,\omega)$. The $\omega$-dual vector field $Z$ of $\lambda$ is called the Liouville vector field. $Z$ satisfies $L_Z\omega=\omega$ and hence its flow is conformally symplectically expanding. \\

Every Liouville domain $X$ can be completed in the following way.Set \[\hat{X}=X\cup(\partial X\times [0,\infty))\] and extend $\lambda$ on $\hat{X}$ as $e^s(\lambda_{\mid \partial X})$ on the attached end. Given a Liouville domain $\mathcal{L}=(X,\omega,\lambda)$ consider the compact set \[Core(\mathcal{L})=\cap_{t>0}Z^{-t}(X)\]It is called Core or the Skeleton of the Liouville domain.\\

Let $\lambda_0$ and $\lambda_1$ be two Liouville forms on a fixed symplectic manifold $(X,\omega)$, moreover let $Z_0$ and $Z_1$ be the respective Liouville vector fields. Then obviously $\lambda_1=\lambda_0+dh$ for some $h:X\to \mathbb{R}$ and $Z_1=Z_0+Z_h$ where $Z_h$ is the hamiltonian vector field for $h$.\\

A Liouville cobordism $(W,\omega,Z)$ is a cobordism $W$ with an exact symplectic form $\omega$ such that the Liouville vector field $Z$ points inward along $\partial_{-}W$ and outward along $\partial_{+}W$.\\

\begin{remark}
\label{completed liouville}
On the infinite end of $\hat{X}$, the Liouville vector field is given by $\partial_s$ irrespective of the choice of the Liouville form $\lambda$ on $X$.
\end{remark}

Now we shall define the Weinstein structures. For this we need to recall few notions. A complete vector field is a vector field whose flow exists for all forward and backward time.\\

Let $\phi$ be a Morse function. A vector field $X$ is called gradient-like for $\phi$ if it satisfies \[X.\phi\geq \delta (|X|^2+|d\phi|^2)\] for some $\delta >0$ and $|X|$ is with respect to some Riemannian metric and $|d\phi|$ is with respect to its dual metric.\\

\begin{definition}(\cite{Kai})
\label{Weinstein-def}
A Weinstein manifold $(X,\omega,Z,\phi)$ is a symplectic manifold $(X,\omega)$ with a complete Liouville vector field $Z$ which is gradient like with respect to the exhausting Morse function $\phi$. A Weinstein cobordism $(W,\omega,Z,\phi)$ is a Liouville cobordism $(W,\omega,Z)$ whose Liouville vector field $Z$ is gradient-like with respect to a Morse function $\phi$ which is constant on the boundary. A Weinstein cobordism with $\partial_{-}W=\Phi\ (empty)$ is called a Weinstein domain.
\end{definition}

In \cite{Eliashberg} Eliashberg has asked the following question.\\

{\bf Problem:} Let $(\hat{X},\lambda_0,\phi_0)$ and $(\hat{X},\lambda_1,\phi_1)$ be two completed Weinstein structures on the same symplectic manifold $(\hat{X},\omega)$. Are they homotopic as Weinstein structures ?\\

Obviously if $Z_0$ and $Z_1$ are the respective Liouville vector fields then $Z_1=Z_0+Z_h$ for $h$ satisfying $\lambda_1=\lambda_0+dh$ and hence \[Z_t=Z_0+Z_{th}=Z_0+tZ_{h},\ t\in [0,1]\] gives a homotopy of Liouville vector fields. However $(X,\omega,Z_t)$ may not be a Liouville homotopy. We refer the reader \cite{Kai} for a precise definition of Liouville homotopy.\\

On Weinstein cobordisms a similar result has been proved in \cite{Kai} although the Weinstein structures need to flexible. We refer the reader to \cite{Kai} for a precise definition of flexible Weinstein structures.

\begin{theorem}(\cite{Kai})
\label{flexible}
Let $(W,\omega_0,\lambda_0,\phi_0)$ and $(W,\omega_1,\lambda_1,\phi_1)$ be two flexible Weinstein structures on the same cobordism $W$ with dimension $2n>4$ which coincide on $Op(\partial_{-}W)$. Let $\eta_t$ be a homotopy rel $Op(\partial_{-}W)$ of non-degenerate two forms on $W$ connecting $\omega_0$ and $\omega_1$. Then there exists a homotopy of flexible Weinstein structures connecting the given ones.
\end{theorem}

Let us now return to the question asked by Eliashberg. We assume that all the zeros of $Z_t$ are non-degenerate for all $t\in [0,1]$. So the zeros of $Z_t$ executes curves $\gamma_i(t),\ i=1,...,k$ (say). We consider $\tilde{X}=\hat{X}\times [0,1]$ and define vector field $Z(x,t)=Z_t(x)$ on $\tilde{X}$. The curves $\gamma_i$'s define curves $\Gamma_i$'s on $\tilde{X}$ as follows \[\Gamma_i(t)=(\gamma_i(t),t)\] Consider two tubular neighborhoods of $\Gamma_i$ as $\Gamma_i\subset N'_i\subset N''_i$. Let $\Psi_i:\tilde{X}\to \mathbb{R}$ be cutoff functions such that $\Psi_i=1\ on\ N'_i$ and $\Psi_i=0\ outside\ N''_i$. Define $\tilde{Z}$ on $\tilde{X}$ by canonically removing the zeros of $Z$ as follows. Define $\tilde{Z}$ close to $\Gamma_i$ as  \[\tilde{Z}(x,t)=\Psi_i(x,t)\partial_t+(1-\Psi_i(x,t))Z(x,t)\] Let $\tilde{\mathcal{F}}$ be the foliation defined by $\tilde{Z}$. Then $\tilde{\mathcal{F}}$ is a regular foliation.

\begin{definition}
\label{Liou-Uni-Open} 
We call the homotopy of the Liouville vector field $Z_t$  uniformly open if it satisfies 
\begin{enumerate}
\item All the zeros of $Z_t$ are non-degenerate for all $t\in [0,1]$\\
\item The foliation $\tilde{\mathcal{F}}\times \tilde{\mathcal{F}}$ on $\tilde{X}\times \tilde{X}$ is uniformly open\\
\end{enumerate}
\end{definition}

Please see \ref{Uniform open} bellow for the definition of Uniformly open foliation. Now we state the main theorem of this paper.

\begin{theorem}
\label{Main}
Let $(\hat{X},\lambda_0,\phi_0)$ and $(\hat{X},\lambda_1,\phi_1)$ be two completed Weinstein structures on the same symplectic manifold $(\hat{X},\omega)$ and let the homotopy of the Liouville vector field $Z_t$ is uniformly open (\ref{Liou-Uni-Open}), moreover $Z_0$ and $Z_1$ do not have a common zero. Then $(\hat{X},\lambda_0,\phi_0)$ and $(\hat{X},\lambda_1,\phi_1)$ can joined by a homotopy of Weinstein structures for which the underlying symplectic structure $\omega$ remains fixed.
\end{theorem}

\begin{remark}
In \ref{flexible} the underlying symplectic structure is not fixed.
\end{remark}

\section{$h$-Principle} This section does not have any new result, we just recall some facts from the theory of $h$-principle which we shall need in our proof.\\

 Let $X\to M$ be any fiber bundle and let $X^{(r)}$ be the space of $r$-jets of jerms of sections of $X\to M$ and $j^rf:M\to X^{(r)}$ be the $r$-jet extension map of the section $f:M\to X$. If $X=M\times N$ then $X^{(r)}$ is denoted as $J^r(M,N)$. A section $F:M\to X^{(r)}$ is called holonomic if there exists a section $f:M\to X$ such that $F=j^rf$. In the following we use the notation $Op(A)$ to denote a small open neighborhood of $A\subset M$ which is unspecified. \\
 
 Let $\mathcal{R}$ be a subset of $X^{(r)}$. Then $\mathcal{R}$ is called a differential relation of order $r$. $\mathcal{R}$ is said to satisfy $h$-principle if any section $F:M\to \mathcal{R}\subset X^{(r)}$ can be homotopped to a holonomic section $\tilde{F}:M\to \mathcal{R}\subset X^{(r)}$ through sections whose images are contained in $\mathcal{R}$. Put differently,  if the space of sections of $X^{(r)}$ landing into $\mathcal{R}$ is denoted by $Sec \mathcal{R}$ and the space of holonomic sections of $X^{(r)}$ landing into $\mathcal{R}$ is denoted by $Hol \mathcal{R}$ then $\mathcal{R}$ satisfies $h$-principle if the inclusion map $Hol \mathcal{R}\hookrightarrow Sec \mathcal{R}$ induces a epimorphism at $0$-th homotopy group $\pi_0$. $\mathcal{R}$ satisfies parametric $h$-principle if $\pi_k(Sec \mathcal{R}, Hol \mathcal{R})=0$ for all $k\geq 0$.\\ 
 
 Let $p:X\to M$ be a fiber bundle and by $Diff_MX$ we denote the fiber preserving diffeomorphisms $h_X:X\to X$, i.e, $h_X\in Diff_MX$ if and only if there exists diffeomorphism $h_M:M\to M$ such that the following diagram commutes

 \[
\xymatrix@=2pc@R=2pc{
& X\ar@{->}[rr]^{h_X} \ar@{->}[d]_{p} & & X \ar@{->}[d]^{p}\\
& M\ar@{->}[rr]^{h_M} & & M\\
}
\]

Let $\pi:Diff_MX\to Diff M$ be the projection $h_X\mapsto h_M$. We call a fiber bundle $p:X\to M$ natural if there exists a homomorphism $j:Diff M\to Diff_MX$ such that $\pi \circ j=id$. For a natural fiber bundle $p:X\to M$ the associated jet bundle $X^{(r)}\to M$ is also natural. The lift is given by \[j^r:Diff M\to Diff_MX^{(r)},\ h\mapsto h_*\] where $h_*(s)=J^r_{j(h)\circ \bar{s}}(h(m))$, $s\in X^{(r)}$, $m=p^r(s)\in M$ and $\bar{s}$ is a local section near $m$ which represents the $r$-jet $s$. Observe $(h^{-1})_*=(h_*)^{-1}$ and hence define $h^*=h_*^{-1}$.\\

For a natural fiber bundle $X\to M$, a differential relation $\mathcal{R}\subset X^{(r)}$ is called $Diff M$-invariant if the action $s\mapsto h_*s,\ h\in Diff M$, leaves $\mathcal{R}$ invariant.\\

\begin{theorem}(\cite{Gromov})
\label{gromov}
If a relation $\mathcal{R}$ is open and $Diff M$-invariant on an open manifold $M$ then it satisfies parametric $h$-principle.
\end{theorem}

\section{Bertelson's Uniformly Open Foliations} In this section we recall some result from \cite{Bertelson} and \cite{Bertelson1}.\\

\begin{definition}(\cite{Bertelson})
\label{Uniform open}
A foliated manifold $(M,\mathcal{F})$ is called uniformly open if there exists a function $f:M\to [0,\infty)$ such that 
\begin{enumerate}
\item $f$ is proper,\\
\item $f$ has no leafwise local maxima,\\
\item $f$ is $\mathcal{F}$-generic.
\end{enumerate}
\end{definition} 

\begin{remark}
\label{*}
Observe that if $dim \mathcal{F}=1$ then $(M,\mathcal{F})$ can not be uniformly open as on a one dimensional manifold, a critical point will be either a local maximum or minimum. 
\end{remark}

So let us explain the notion $\mathcal{F}$-generic. In order to do so we need to define the singularity set $\Sigma^{(i_1,i_2,...,i_k)}(f)$ for a map $f:M\to W$. $\Sigma^{i_1}(f)$ is the set \[\{p\in M:dim(ker(df)_p)=i_1\}\] It was proved by Thom \cite{Thom} that for most maps $\Sigma^{i_1}(f)$ is a submanifold of $M$. So we can restrict $f$ to $\Sigma^{i_1}(f)$ and construct $\Sigma^{(i_1,i_2)}(f)$ and so on. In \cite{Thom} it has been proved that there exists $\Sigma^{(i_1,...,i_k)}\subset J^k(M,W)$ such that $(j^kf)^{-1}\Sigma^{(i_1,...,i_k)}=\Sigma^{(i_1,...,i_k)}(f)$.\\

 Let us set $W=\mathbb{R}$ as this is the only situation we need. Let $(M,\mathcal{F})$ be a foliated manifold with a leaf $F$. Define the restriction map \[r_F:J^k(M,\mathbb{R})\to J^k(F,\mathbb{R}):j^kf(x)\mapsto j^k(f_{\mid F})(x)\] Define foliated analogue of the singularity set as \[\Sigma^{(i_1,i_2,...,i_k)}_{\mathcal{F}}:=\cup_{\{F\ leaf\ of\ \mathcal{F}\}} r_F^{-1}\Sigma^{(i_1,i_2,...,i_k)}\]

\begin{definition}(\cite{Bertelson})
A smooth real valued function $f:M\to \mathbb{R}$ is called $\mathcal{F}$-generic if the first jet $j^1f \pitchfork \Sigma^{(n)}_{\mathcal{F}}$ and the second jet $j^2f \pitchfork \Sigma^{(i_1,i_2)}_{\mathcal{F}}$ for all $(i_1,i_2)$.
\end{definition}

\begin{definition}(\cite{Bertelson})
\label{Foliated invariant}
An isotopy of the manifold $M$ is a family $\psi_t,\ t\in [0,1]$ of diffeomorphisms of $M$ such that the map $\psi:M\times [0,1]\to M\ :(x,t)\mapsto \psi_t(x)$ is smooth and $\psi_0=id_M$. Consider a foliation $\mathcal{F}$ on $M$. A foliated isotopy of $(M,\mathcal{F})$ is an isotopy $\psi_t$ of $M$ that preserves the foliation $\mathcal{F}$, that is, $(\psi_t)_*(T\mathcal{F})=T\mathcal{F}$ for all $t\in [0,1]$. A relation $\mathcal{R}$ is called foliated invariant on $(M,\mathcal{F})$ if the action by foliated isotopies leaves $\mathcal{R}$ invariant.
\end{definition}

\begin{theorem}(\cite{Bertelson})
\label{bertelson}
On an uniformly open foliated manifold, any open, foliated invariant differential relation satisfies the parametric $h$-principle.
\end{theorem}

In \cite{Bertelson1} Bertelson has contructed counter examples that without the uniformly open condition \ref{bertelson} fails.\\

\section{Main Theorem} In this section we prove \ref{Main}. Let us first set some notations. First of all we have the Liouville vector fields $Z_0$ and $Z_1=Z_0+Z_h$ and let $Z_t=Z_0+tZ_1$ be the homotopy of uniformly open Liouville vector field. So we have 

\begin{enumerate}
\item $Z_0.\phi_0\geq \delta (|Z_0|^2+|d\phi_0|^2)$\\
\item $Z_0.\phi_1+Z_h.\phi_1\geq \delta' (|Z_0+Z_h|^2+|d\phi_1|^2)$\\
\end{enumerate}

With equality occurs in the above inequalities at the zeros of $Z_0$ and $Z_0+Z_h$. Define \[\phi_t=(1-t)\phi_0+t\phi_1\] Then observe that \[Z_0.d\phi_t=(1-t)Z_0.d\phi_0+tZ_0.d\phi_1\] Now consider 
\[
\begin{array}{rcl}
Z_0.d\phi_t+Z_{th}.d\phi_1 &=& (1-t)Z_0.d\phi_0+t[Z_0.d\phi_1+Z_h.d\phi_1]\\
 &\geq & (1-t)\delta (|Z_0|^2+|d\phi_0|^2)+t\delta' (|Z_0+Z_h|^2+|d\phi_1|^2)\\
 &\geq & min(\delta,\delta')[(1-t)|Z_0|^2+t|Z_0+Z_h|^2+(1-t)|d\phi_0|^2+t|d\phi_1|^2]
\end{array}
\]
 So we get 
 
 \[
 \begin{array}{rcl}
 \frac{(Z_0.d\phi_t+Z_{th}.d\phi_1)}{|Z_0+tZ_h|^2+|d\phi_t+d\phi_1|^2} &\geq & min(\delta,\delta')[\frac{(1-t)|Z_0|^2+t|Z_0+Z_h|^2}{(1-t)^2|Z_0|^2+t^2|Z_0+Z_h|^2+2t(1-t)|Z_0||Z_0+Z_h|+|d\phi_t+d\phi_1|^2}\\
  & & +\frac{(1-t)|d\phi_0|^2+t|d\phi_1|^2}{(1-t)^2|Z_0|^2+t^2|Z_0+Z_h|^2+2t(1-t)|Z_0||Z_0+Z_h|+|d\phi_t+d\phi_1|^2}]
 \end{array}
 \]
 
 Recall that (according to \ref{completed liouville}) on the infinite end of $\hat{X}$ the Liouville vector fields $Z_0$ and $Z_0+Z_h$ are equal to  $\partial_s$. Moreover since $Z_0$ and $Z_1$ do not have a common zero $|d\phi_t+d\phi_1|>0$ and since $\Gamma_i$'s are compact $|d\phi_t+d\phi_1|$ is bounded bellow.\\

  So the right hand side of the above inequality is bounded and hence the right hand side is equal to $\tilde{\delta}$ (say). So we get \[(Z_0.d\phi_t+Z_{th}.d\phi_1)\geq \tilde{\delta} [|Z_0+tZ_h|^2+|d\phi_t+d\phi_1|^2]\]
 
 Without loss of generality we assume that $Z_0\pitchfork Z_h$ otherwise we can use a relative version of $h$-principle.\\
 
  We replace $t$ by a new parameter $t'=f(t)$ where $f:[0,1]\to [0,1]$ is such that $f=0$ on $[0,\epsilon]$ and $f=1$ on $[1-\epsilon,1]$. We can replace the parameter in the above inequality.\\
 
 Define one forms $\alpha_{t'}$ and $\eta_{t'}$ as follows. First $\alpha_{t'}$, $\alpha_{t'}(Z_0)=d\phi_{t'}(Z_0)$,  $Z_{h}\in ker(\alpha_{t'}),\ for\ t\in [\epsilon,1-\epsilon],\ t'=f(t)$ and $\alpha_0=d\phi_0,\ \alpha_1=d\phi_1$. Similarly $\eta_{t'}(Z_{t'h})=\eta_{t'}(t'Z_h)=d\phi_1(t'Z_h)$, $Z_0\in ker\eta_{t'}\ for\ t\in [\epsilon,1-\epsilon],\ t'=f(t)$ and $\eta_0=d\phi_1=\eta_1$. So we have \[(\alpha_{t'}+\eta_{t'})(Z_0+t'Z_h)=(\alpha_{t'}(Z_0)+\eta_{t'}(Z_{t'h}))\geq \tilde{\delta} [|Z_0+t'Z_h|^2+|\alpha_{t'}+\eta_{t'}|^2]\] 
Now extending on $\tilde{X}$ and regularizing $Z_{t'}=Z_0+t'Z_h$ we get $\tilde{Z}$ as in \ref{1}. We extend $\alpha_{t'}$ and $\eta_{t'}$ to $\tilde{X}$ as $\alpha'(x,t')=\alpha_{t'}(x)\ and\ \eta'(x,t')=\eta_{t'}(x)$. Adjusting $\alpha'$ and $\eta'$ near $\Gamma_i$'s (\ref{1}) to $\tilde{\alpha}$ and $\tilde{\eta}$ so that \[(\tilde{\alpha}+\tilde{\eta})(\tilde{Z})>\tilde{\delta}[|\tilde{Z}|^2+|\tilde{\alpha}+\tilde{\eta}|^2]\]

Now we come to the $h$-principle part. Consider $M=\tilde{X}\times \tilde{X}$ and the trivial bundle $P\times P:M\times \mathbb{R}^2=\tilde{X}\times \mathbb{R}\times \tilde{X}\times \mathbb{R}\to M$ where $P:\tilde{X}\times \mathbb{R}\to \tilde{X}$ is the projection on the first factor. Observe that \[(M\times \mathbb{R}^2)^{(1)}=(\tilde{X}\times \mathbb{R})^{(1)}\times (\tilde{X}\times \mathbb{R})^{(1)}\] Note that this does not happen in case of higher order jet extensions as there will be mixed derivatives.\\

Observe that the section space $\Gamma(\tilde{X}\times \mathbb{R})=C^{\infty}(\tilde{X},\mathbb{R})$. There is a natural affine fibration $L:(\tilde{X}\times \mathbb{R})^{(1)}\to T^*(\tilde{X}\times \mathbb{R})$ given by $L(j^1f(x))=df_x$ where $f\in C^{\infty}(\tilde{X},\mathbb{R})$. Define the relation $\mathcal{R}\subset (M\times \mathbb{R}^2)^{(1)}$ as \[\mathcal{R}=\{(j^1f_0,j^1f_1)\in (M\times \mathbb{R}^2)^{(1)}:L(j^1f_i)(\tilde{Z})>\tilde{\delta}[|\tilde{Z}|^2+|L(j^1f_i)|^2]\ for\ i=0,1,\ for\ some\ \tilde{\delta}>0\}\] Obviously $(\tilde{\alpha}+\tilde{\eta},\tilde{\alpha}+\tilde{\eta})\in Sec\mathcal{R}$. Next we shall show that $\mathcal{R}$ is open and invariant under $\tilde{\mathcal{F}}\times \tilde{\mathcal{F}}$-foliated isotopy. This will conclude the proof of \ref{Main} in view of \ref{bertelson}. Only thing one needs to do is the following. Let $(f_0,f_1)$ is a resulting solution. Choose either $f_0$ or $f_1$ say $f_0$. Then define $f_t$ as \[f_t(x)=f_0(x,t)\] Now we have to re-introduce the singularities. Let $g_t$ be a family of Morse functions defined near $\Gamma_i$ with index same as the index of $Z$ along $\Gamma_i$. Let $\beta$ be a cutoff function such that $\beta=1$ on a tubular neighborhood $\mathcal{N}_i \supset N''_i$ and $\beta=0$ outside $\mathcal{N}'_i\supset \mathcal{N}_i$.Let $\beta_t(x)=\beta(x,t)$. Observe \[Z_t(\beta_tg_t+(1-\beta_t)f_t)=[\beta_tZ_t(g_t)+(1-\beta_t)Z_t(f_t)]+g_tZ_t(\beta_t)-f_tZ_t(\beta_t)\] Observe that $Z_t(\beta_t)$ has compact support and $g_t$ is of the form \[a+x_1^2+...+x_k^2-x_{k+1}^2+...+x_{2n}^2\] So if we take $a$ large enough then $(g_tZ_t(\beta_t)-f_tZ_t(\beta_t))>0$ and obviously compactly supported. So we get the desired result.\\

\begin{lemma}
The relation $\mathcal{R}$ is open and invariant under the action of $\tilde{\mathcal{F}}\times \tilde{\mathcal{F}}$-foliated isotopies.
\end{lemma}

\begin{proof}
Openness of $\mathcal{R}$ follows directly from the definition of $\mathcal{R}$.\\

For second part we see $\psi_s^*(df)(\tilde{Z})=df(d\psi_s(\tilde{Z}))\geq cdf(\tilde{Z})$, where $c$ is a positive real number. Positive as $\psi_0=id$ and $M\times\mathbb{R}^2$ is connected. So \[\psi_s^*(df)(\tilde{Z})\geq cdf(\tilde{Z})>c\tilde{\delta}[|\tilde{Z}|^2+|df|^2]\]  
\end{proof}

\end{document}